\documentclass[leqno,12pt]{article}
\usepackage[margin=1.3in, top = 1.3in, bottom=1.3in]{geometry}
\usepackage[utf8]{inputenc}
\usepackage{amsmath}
\usepackage{amsfonts}
\usepackage{mathrsfs}
\usepackage{tikz}
\usepackage{stmaryrd}
\usepackage{enumerate}
\usepackage{tikz-qtree}
\usetikzlibrary{automata,positioning}
\usepackage{amssymb}
\usepackage{amsthm}
\usepackage{bbm}
\usepackage{mathtools}
\usepackage{algorithm}
\usepackage{algpseudocode}
\usepackage{diffcoeff}
\usepackage{interval}
\usepackage{setspace}
\usepackage{graphicx}
\usepackage{comment}
\usepackage{fancyhdr}
\usepackage{tikz}
\usepackage{tikz-3dplot}
\usepackage{enumitem}
\usepackage{hyperref}
\hypersetup{hidelinks}
\usetikzlibrary{shapes,arrows}
\theoremstyle{definition}
\newcommand{\R}{\mathbb{R}}
\newcommand{\st}[1]{\left\{#1\right\}}

\newcommand{\N}{\mathbb{N}}

\newcommand{\from}{\colon}
\newcommand{\seq}{\subseteq}

\newcommand{\clos}{\overline}

\DeclareMathOperator{\conv}{conv}

\newtheorem{theorem}{Theorem}[]

\newtheorem{corollary}{Corollary}[]
\newtheorem{proposition}{Proposition}[]
\newtheorem{lemma}[]{Lemma}

\tikzstyle{decision} = [diamond, draw, fill=blue!20, 
text width=6em, text badly centered, node distance=3cm, inner sep=0pt]
\tikzstyle{block} = [rectangle, draw, fill=blue!20, 
text width=9em, text centered, rounded corners, minimum height=4em]
\tikzstyle{block2} = [rectangle, draw, fill=green!20, 
text width=9em, text centered, rounded corners, minimum height=4em]
\tikzstyle{line} = [draw, -latex']
\tikzstyle{cloud} = [draw, ellipse,fill=red!20, node distance=3cm,
minimum height=4em]

\counterwithin{figure}{section}

\begin{document}
\title{\vspace{-2.5cm} A finite set with noncompact closed convex hull in a Hadamard space}
\author{Ariel Goodwin
	\thanks{Center for Applied Mathematics, Cornell University, Ithaca, NY, USA.
	\texttt{awg77@cornell.edu}  Research supported in part by the NSERC Postgraduate Fellowship PGSD-587671-2024.}
}
 
\date{September 29, 2026}

	\maketitle

	\begin{abstract}
		We show that there is a Hadamard space in which the closed convex hull of a set containing merely three points  
		is noncompact, by providing an AI-generated construction. The seemingly innocuous problem of finding a compact or finite set with this property, or showing that none 
		exists, was recorded by Gromov 
		and periodically revisited by researchers in metric geometry, functional analysis, fixed-point theory, and geometric group theory.
		Nevertheless, it remained unanswered for over three decades. 
		The construction itself is based on building an inductive sequence of CAT(1) metric graphs and taking 
		the Euclidean cone over the limiting space. We discuss the history of this problem,
		the main ideas and details of the construction, and its relationship with existing techniques.
	\end{abstract}

\medskip

\medskip

	\section{Introduction}
	Complete geodesic metric spaces of nonpositive curvature, in the CAT(0) sense, 
	are commonly known as Hadamard spaces. The main examples of CAT(0) spaces are Euclidean and Hilbert spaces, Hadamard manifolds (complete, simply-connected, Riemannian manifolds
	of nonpositive sectional curvature) such as hyperbolic space and the cone of positive-definite matrices with its affine-invariant metric \cite{Boumal2023}, 
	metric trees \cite{Bridson1999, Handler1973}, and CAT(0) cubical complexes \cite{Ardila2012}.
	Such spaces, though seemingly abstract at first glance, are characterized by the strong convexity of their squared distance functions, and 
	thus share many similarities with Euclidean and Hilbert spaces when it comes to questions of analysis, geometry, convexity,
	and optimization \cite{Bacak2014}. At the same time, there are some notable differences between Euclidean convexity and geodesic convexity
	in Hadamard spaces that challenge our Euclidean intuition. Examples of these phenomena include the 
	pathological behavior of certain classical optimization algorithms \cite{LytchakPetrunin2022, criscitiello22},  
	the existence of three-point sets whose convex hull has dimension greater than two and thus differs from the set of weighted Fr\'echet means \cite{LubiwMaftuleacOwen2020ConvexHulls, 
	LinEtAl2017ConvexityTreeSpaces, bessenyei2021sandwich,Goodwin2025,
	alexander2024alexandrov,kleiner1999local}, and a surprising discrepancy between 
	inner and outer characterizations of convexity \cite{borisenko1999total,santalo1972averages} with ramifications to the convergence analysis
	of certain subgradient-like methods \cite{GoodwinSubgrad2026,
	 goodwin2026circumcenters,zhang2016firstorder}.
	
	Straddling the boundary of what we know about geometry in Hadamard spaces is a fundamental question about the compactness of convex hulls.
	Specifically, it has remained unknown whether the closed convex hull of a compact --- or even finite --- set is itself compact. 
	It was known 
	to be true for all
	compact sets if and only if it was true for all finite sets \cite[Lemma 2.18]{duchesne2016groups}. 
	This apparent gap in our understanding
	contrasts strongly with its positive and essentially elementary resolution in Hilbert (and Banach) spaces \cite[Exercise 1.62]{fabian2011banach}.
	To our knowledge, the question in Hadamard spaces dates back to Gromov \cite[6.B$_1$(f)]{Gromov1993}, 
	has been frequently revitalized by Petrunin \cite{petrunin2009pigtikal,LytchakPetrunin2022,Petrunin2025Gallery} (attributing it to mathematical
	folklore predating Gromov's \cite[6.B$_1$(f)]{Gromov1993}), and has more recently 
	been highlighted in a comprehensive survey by Ba\v{c}\'a{k} \cite{Bacak2023}. In each of 
	\cite{Monod2016,BacakMeans2014,LytchakPetruninConvex2022,KopeckaReich2007,
	Berdellima2020Investigations,Berdellima2021CompactSets, 
	Goodwin2025,Basso2025SomeQuestions, LytchakPetrunin2023WeakTopology,duchesne2016groups} it is noted that the 
	compactness of a closed convex hull in a Hadamard space remains unclear, and occasionally speculated that there could be
	noncompact examples. 
	Spaces in which the result is true are sometimes referred to as spaces with property (CH): see
	\cite{jost1994equilibrium} where this property has implications on the regularity of energy minimizers in Hadamard spaces.
		
	Understanding the behavior of convex hulls in nonlinear spaces 
	is a delightfully delicate endeavor. In \cite{Monod2016}, Monod shows that there exists a bounded Hadamard space
	with no extreme points, and thus a natural analogue of the Krein-Milman theorem fails in the Hadamard setting. Lytchak and Petrunin in
	\cite{LytchakPetruninConvex2022} demonstrate that, generically, the convex hull of three points in a Riemannian manifold
	has nonempty interior. For two-dimensional CAT(0) cubical complexes there are some positive results 
	on the compactness of convex hulls
	\cite{Burnell2006CompactnessConvexHulls}, and algorithms for computing 
	convex hulls of finite sets \cite{LubiwMaftuleacOwen2020ConvexHulls}.
	Basso, Krifka, and Soultanis recently showed in \cite{BassoKrifkaSoultanis2024NoncompactHull} that if one relaxes the notion of nonpositive curvature from
	CAT(0) to that induced by a conical geodesic bicombing, the closed convex hull of finitely many points need not be compact. 
	That brings us to this paper, in which we show:
	
	\begin{center}
		\textbf{Theorem:} \textit{There is a Hadamard space containing a set of three points with noncompact closed convex hull.}
	\end{center}
	Since the convex hull of two points is the geodesic between them, we see at last that the 
	convex hull of a finite set in a Hadamard space can exhibit pathological behavior in the smallest
	nontrivial case.
	In the first version of this paper we presented a weaker version of the theorem, discovered by AI, providing only a finite set with 
	large and unspecified cardinality. 
	The improvement to three points was found shortly thereafter, again with AI.
	The main idea in both examples is exactly the same: to consider an inductively defined metric graph with cleverly 
	chosen edge-lengths and no cycles shorter than $2\pi$, then pass to the
	Euclidean cone over this metric graph to obtain a CAT(0) space from a CAT(1) space.
	The first example used a needlessly 
	strong result from graph theory to instantiate a graph with a 
	large number of vertices and no short cycles \cite{ErdosSachs1963}. It was the
	author's original intent to include both the first example and the three-point 
	example in the same manuscript to better illustrate the process of discovery. However, after seeing 
	how much the presentation 
	could be simplified by using a different and explicit initial graph, it was decided to discuss only the three-point example in this 
	paper and leave the first example archived for posterity \cite{goodwin2026finitesetnoncompactclosed}.
	A comparison with the earlier version will reveal that the strategy remains completely identical, just
	using slightly different arguments to show that the initial graph guarantees
	infinite growth.
	In parallel with the
	author's preparation of the three-point refinement, Lytchak also upgraded the finite set
	example to a three-point example in \cite{lytchak2026noncompact}, sharing many common features with the construction 
	presented here\footnote{Our decision to
	continue writing this updated version of \cite{goodwin2026finitesetnoncompactclosed} to include the three-point example, which we had discovered and had been working on
	writing up when Lytchak's paper appeared,
	was encouraged in communication with Lytchak.}.
	 Lytchak's manuscript also mentions
	several interesting topological and geometric properties of the space in which the counterexample lives.

	In the rest of the paper we walk through the solution to this problem, found using AI, showing that there
	is a Hadamard space such that the closed convex hull of a particular three-point set is not compact. This problem, with its nearly elementary statement, has provided a 
	rich source of interest for mathematicians. To inject some personal sensibility, there are a few aspects that make the original question and its
	resolution so striking. For one, convexity is a focal point for many diverse areas of pure and applied mathematics, including the author's own research interests.
			  Geodesic convexity is the obvious generalization of convexity to nonlinear spaces, and in many respects 
			  Hadamard spaces are very well-behaved when it comes to the study of geodesically convex sets and functions. But knowing that the convex hull of a mere three points in a Hadamard space
			  can behave so unintuitively provides a significant data point in support of seeking alternative notions of convexity that 
			  are closer to our Euclidean ideals. It also invites questions about what additional structure besides
			  nonpositive curvature,
			  such as smoothness, can be imposed to ensure that convex hulls behave as they do in linear spaces.

		Furthermore, the problem statement itself is extremely simple and provided enough leeway in either direction
		 that experts in metric geometry were unsure which way it might resolve. 
		 To paint with broad strokes, in classical convex analysis
		we are spoiled by theorems about convex objects having simple and elegant proofs. When things seem to 
		break down, there are a handful of usual counterexamples to guide our thinking 
		(though we are still finding more in recent years \cite{BoltePauwels2022}).  
		The example presented in this paper serves to expand our toolbox of familiar cases upon which to test questions about convex analysis and
		optimization in nonlinear spaces.

		Lastly, the example is nontrivial but highly accessible. It is clear that it employs a careful combination of 
			  ideas, decisions, and parameters, which in the author's opinion are impressive and surprising. At the same time, the tools used 
			  are completely familiar to those in the field, and certainly within reach for a 
			  nonspecialist graduate student interested in geometry. The techniques and cited results used in constructing the 
			  example either come from elementary graph theory or have long been canonized in standard textbooks on
			  metric geometry such as \cite{Bridson1999, BuragoBuragoIvanov2001}.
			  The author has strived to 
			  formulate and communicate the presented arguments in a lucid way, perhaps bordering occasionally 
			  on being too detailed.
			  This reflects a deliberate choice to make the result approachable for a wide audience.  
	
	In our view, the techniques used 
	are simple and elegant enough to warrant further reflection and internalization. It seems in retrospect that they were not
	far out of reach. We hope that this work provides a fruitful point of departure for the mathematical community to better understand convexity in nonlinear spaces.

	\subsection*{Acknowledgements}
	The author would like to thank Adrian Lewis, Genaro L\'opez-Acedo, and Adriana Nicolae
	for many valuable conversations on the subjects of 
	convex optimization and metric geometry. The author owes
	thanks more broadly to many people whose work in these areas has been both 
	inspirational and educational. This includes, but is not limited to,
	Anton Petrunin, Miroslav Ba\v{c}\'{a}k, Nicholas Pischke, Vitali Kapovitch, Alexander 
	Lytchak, Nicolas Monod, and Elefterios Soultanis, who have been 
	warmly receptive to these recent developments and whose feedback has been most valuable. 

	\section{Preliminaries}
	We recall a few standard notions from metric geometry and graph theory; 
	some authoritative references are \cite{Bridson1999,BuragoBuragoIvanov2001} and
	\cite{diestel2025graph}, respectively.
	A metric space $(X,d)$ is a \textit{geodesic space} if every two points $x,y\in X$ can be joined by a
	\textit{geodesic}, which is to say a map $\gamma$ from a closed interval $[a,b]$ into $X$ with $d(\gamma(t),\gamma(t')) = |t-t'|$
		for all $t,t'\in [a,b]$. A geodesic space is said to be \textit{CAT(0)} if the map 
		\[t\mapsto d(\gamma(t),a)^2-t^2\] is convex for every point $a\in X$ and every geodesic $\gamma$. A \textit{Hadamard space} is a complete CAT(0) space. 
		In a Hadamard space, a geodesic joining two given points is unique. 
		A subset $C$ of a Hadamard 
		space $X$ is \textit{geodesically
		convex} if the geodesic between any two points in $C$ is contained in $C$. A metric space $(X,d)$ is \textit{proper} if the closed ball $B_r(x)$ is compact for every point $x\in X$ 
		and radius $r > 0$.
		The \textit{convex hull} of a set $C$ of a Hadamard space is the smallest geodesically convex set containing 
		$C$, denoted by $\conv(C)$. The \textit{closed convex hull} of such a set $C$ is the smallest closed convex set containing $C$,
		denoted by $\clos{\conv}(C)$. We will also encounter CAT(1) spaces in the construction, but it is sufficient for us to 
		work with such spaces through their properties rather than a formal definition (which can be found in \cite{Bridson1999}). Informally, a CAT(1) 
		space is a geodesic space of curvature at most 1, where curvature is measured by comparing triangles in the space with corresponding
		triangles in the sphere $S^2$ of radius 1. Since CAT(0) signifies curvature at most 0, 
		CAT(0) spaces are also CAT(1). 

	A graph is an ordered pair $G = (V,E)$ of vertices $V$ and edges $E$, and the cardinality of the vertex set $|V|$ is
	called the \textit{order} of $G$. Sometimes we may use the notation $V(G)$ and $E(G)$ to indicate the
	vertex and edge sets, respectively, of a particular graph $G$. A graph is \textit{connected} if any two of its vertices are joined by a finite 
	sequence of edges, called
	a \textit{path}. The \textit{combinatorial distance} between two vertices in a graph is the number of edges in a shortest path connecting them. 
	A vertex $v\in V$ is said to have \textit{degree} $\deg(v)=d$ if it has $d$ neighbors. Sometimes we will use 
	the notation
	$\deg_G$ to specify the degree relative to a specific graph $G$.
	At the intersection of metric geometry and graph theory are spaces called \textit{metric graphs}, which can be understood
	intuitively as connected graphs whose individual edges are metrized by identifying them with compact intervals of the real line. 
	Such metric graphs will be called \textit{finite} if their edge and vertex sets are finite.
	An \textit{edge segment} of an edge in a metric graph refers to
	a subset of that edge which is the image of a compact subinterval of the (potentially larger)
	interval that parametrizes that edge. The edge segment between
	two points $x$ and $y$ belonging to a common edge will be denoted by $[x,y]$ (or with open brackets to exclude endpoints).
	The underlying (pseudo)metric, called the \textit{path length pseudometric},
	on the whole graph is defined by taking infimums over lengths of piecewise linear paths between points in the graph; once more we defer a rigorous presentation of the   
	definition to \cite[I.1.9]{Bridson1999}. Unless specified otherwise, length of a path or an edge
	will refer to this continuous notion of length rather than the combinatorial length that counts
	the number of edges.

	\section{Construction of a noncompact convex hull}

	We briefly describe the construction at a high level.
	The construction starts by considering an inductively defined sequence of metric
	graphs. The inductive procedure adjoins children to certain pairs of 
	parent vertices in the current graph such that 
	\begin{enumerate}[label=(\alph*)]
		\item sets of parents have a unique child in the next step of the construction,
		\item the graph never admits any cycles of length less than $2\pi$, which  
				amounts to maintaining the CAT(1) property (an upper curvature bound of 1), and
		\item each step adds a finite but nonzero number of new vertices and edges.
	\end{enumerate}
	Each such graph is naturally viewed as a subset of the next, 
	so we can make sense of the infinite union of all the graphs in the sequence to obtain a space $Y$. 
	We will show that $Y$ itself is a complete
	CAT(1) metric graph. Then the actual Hadamard space $X$ we will consider is the \textit{Euclidean cone} over this CAT(1) space. We will define the Euclidean
	cone as needed
	in Subsection \ref{finiteset}.
	The Euclidean cone over a complete CAT(1) space is a Hadamard space by a theorem of Berestovskii \cite[Theorem II.3.14]{Berestovskii1983,Bridson1999}. 
	In this step, 
	the edge lengths in the graph $Y$ translate to angles in the cone $X$. The choice of
	slowly vanishing edge lengths in $Y$ pays off by creating a set of 
	points $Q$, placed at specific radii away from the tip of the cone, that is not
	totally bounded. All points in $Q$ arise as iterated geodesic midpoints
	from an initial finite set of points corresponding to a subset of $Y$, which motivates the
	iterated parent-child structure. 
	The paper proceeds in subsections that roughly follow this informal roadmap. 
	Also, the closed convex hull of a compact set in a proper Hadamard space is compact
	for elementary reasons, so the counterexample cannot be proper \cite{Petrunin2009ConvexHull}.

	\subsection{The example of Basso-Krifka-Soultanis}
	Before getting into the details, it is well worth highlighting the nearest predecessor to this construction in the literature, 
	constructed by Basso, Krifka, and Soultanis \cite{BassoKrifkaSoultanis2024NoncompactHull}. 
	Their example is situated in complete metric spaces $(X,d)$ that admit a \textit{conical geodesic bicombing}, which is a family of (linearly reparametrized)
	geodesics
	$\st{\sigma_{xy}}_{x,y\in X}$ (the subscripts $x$ and $y$ are the respective start and end points of a given geodesic $\sigma_{xy}$) 
	that obey a convexity inequality with respect to the metric:
	\[d(\sigma_{xy}(t),\sigma_{x'y'}(t)) \leq (1-t)d(x,x') + td(y,y') ~ ~\text{ for all } x,y,x',y' \in X,~  t\in[0,1].\]
	Hadamard spaces are examples of spaces admitting a conical geodesic bicombing. A set is \textit{$\sigma$-convex} if it is closed under 
	the conical geodesic between any two points within it, and the \textit{$\sigma$-convex hull} is defined accordingly.
	In some sense the earlier example is motivated by a suspicion that there is a finite set in a Hadamard space whose closed convex hull is noncompact: the space constructed in one of their main results, if 
	it were compact, would imply the compactness of closed convex hulls for all finite sets in Hadamard spaces \cite[p.\ 5865]{BassoKrifkaSoultanis2024NoncompactHull}.
	However, they remark that their construction seems ill-suited for
	generating spaces that are actually Hadamard, and not just equipped with a conical geodesic bicombing. 
	
	Certainly, the counterexample in Hadamard space is constructed in a fashion that evokes that of the Basso-Krifka-Soultanis example:
	\begin{enumerate}
		\item Choose an initial set $V_0$ of specified cardinality.
		\item Inductively, construct an infinite sequence of sets
		\[V_0 \rightarrow V_1 \rightarrow V_2 \rightarrow \cdots\]
		by defining ``midpoints" of some subset of the previous points.
		\item Show that a limiting space defined by the sets $\st{V_n}_{n=0}^\infty$ can be made into a metric space $(V,d)$ with certain properties.
		\item Pass from the space $(V,d)$ to a nonpositively curved metric space $(X,d')$.  
		\item Exhibit a subset of points in $(X,d')$, derived from midpoints in $(V,d)$, that is not totally bounded. This is achieved by the way $V$ is constructed.
		\item Finally, the initial set $V_0$ induces a subset of $X$ whose convex hull contains the unwieldy set of midpoints described in the previous step.
	\end{enumerate}
	On the other hand, there are some notable differences.
	\begin{enumerate}

		\item Much of the geometric structure in the example of \cite{BassoKrifkaSoultanis2024NoncompactHull} 
		is imposed in a ``top-down'' fashion, as opposed to a ``bottom-up'' approach in our presentation.  
		For example, the metric $d$ on $V$ and its conical geodesic bicombing on the completion of $(V,d)$ are determined at the end of the 
		construction. We operate instead with a sequence of geodesic metric spaces that are constructed with 
		carefully chosen midpoints to inductively preserve a curvature bound. 

		\item Basso-Krifka-Soultanis's construction starts with a preliminary metric space defined from repeated midpoints
		and turns it into a space admitting a conical geodesic bicombing 
		by taking the metric completion.  
			  The construction we present differs in that the space obtained from repeatedly adjoining midpoints is CAT(1), and we get the CAT(0) space
			  of interest by moving to the Euclidean cone over the space. This is a more transformative operation than completion.
	\end{enumerate}
	
	\subsection{Building metric graphs with long cycles} 
	\label{longcycle}

	Let us introduce a sequence of positive parameters $\st{\alpha_n}_{n=0}^\infty$
	which will play a geometric role in the construction. Initially, they will serve as edge lengths in a growing
	sequence of metric graphs. Later, after passing to the Euclidean cone over this limiting space, these edge lengths 
	will be interpreted rather as angles. In any case, we will choose these parameters explicitly as follows. Fix the integer $M = 200$ and
	set 
	\[\alpha_0 = \frac{M\pi}{2M+1}, \quad \quad \alpha_n = \frac{M\pi}{2M + n} ~ ~ (n\geq 1).\]
	Our choice of $M$ is somewhat arbitrary and it will later be clear that we just need $M$ to be sufficiently large
	to ensure that the inductive construction never terminates. 
	The main properties we will need from the sequence of edge lengths are that it is nonincreasing and
	\[0 < \alpha_n < \frac{\pi}{2}, \quad \sum_{n=0}^\infty\alpha_n = \infty, \quad \sum_{n=0}^\infty \alpha_n^2 < \infty.\]	
	We further instantiate variables for the partial sums of the sequence:
	\begin{equation}
		\label{scalars}
		s_0 = 0, \quad s_n = \sum_{k=0}^{n-1}\alpha_k ~ ~ (n\geq 1).
	\end{equation}
	In particular, $s_n\to \infty$ as $n\to \infty$ by the nonsummability of $\st{\alpha_n}_{n=0}^\infty$.

	First, we construct a sequence of metric graphs whose limit, in some sense, is a complete CAT(1) space.
	The initial graph in this sequence is obtained by choosing a graph of sufficiently large order.	
	The construction used to prove the next theorem proceeds in inductive generations by adding
	vertices and edges to an existing metric graph to obtain the next metric graph in 
	the sequence. 
	In each generation, any newly added vertex is connected to exactly two vertices in the preceding
	generation. We will refer to the two preceding vertices as its \textit{parents}, and to 
	the vertex born of a given pair of parents as their \textit{child}.

	Begin with a set of $2M = 400$ vertices $V_0 = \st{a_1,\dots,a_{2M}}$ connected by edges
	of length $\delta = \frac{5\pi}{4M}$ between $a_i$ and $a_{i+1}$ for $i=1,\dots,2M$ (mod $2M$).
	One can interpret this as a circle $C$ of circumference $5\pi/2$ and $2M$ equally spaced points
	along it. We will refer back to this cycle $C$ as the (initial) circle.
	Now add $M$ children $V_1 = \st{b_1,\dots,b_M}$ 
	such that $b_i$ is the child of $a_i$ and $a_{i+M}$, with connecting edges of
	length $\alpha_0$. This gives $2M$ new edges, which can be viewed as $M$ subdivided chords of the circle 
	$C$. We call $(G_0,d_0)$ the metric graph obtained so far, with vertex set $V(G_0) = V_0 \cup V_1$,
	all circle edges having length $\delta$, and all parent-child edges having length $\alpha_0$.
	This initial graph is illustrated in Figure \ref{initgraph} below.
		\begin{figure}[h!]
			\begin{center}
			\includegraphics[scale=0.8]{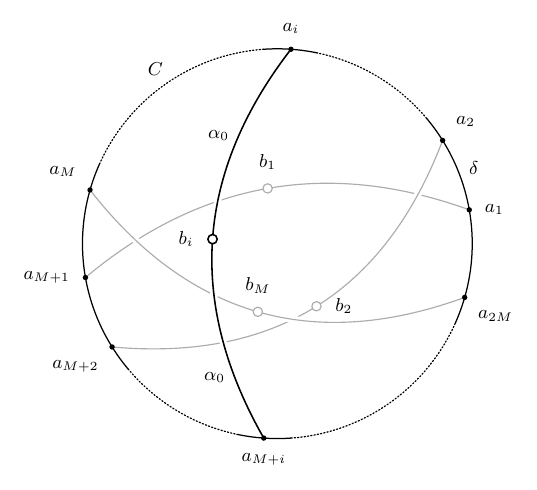}
			\caption{The initial metric graph $(G_0,d_0)$.}
			\label{initgraph}
			\end{center}
		\end{figure}
	\begin{theorem}
		\label{graphsequence}
		There exists a sequence of connected finite metric graphs $\st{(G_n,d_n)}_{n=0}^\infty$, where $d_n$ is the usual shortest path metric
		on $G_n$ for each $n\geq 0$,
		 with the following properties:
		\begin{enumerate}[label=(\alph*)]
			\item $G_0$ is the metric graph specified in the preamble to this theorem.
			\item $G_{n+1}$ is obtained from $G_n$ by the addition of new child vertices, each joined by an edge to 
			each of its two parents in $G_n$, and 
			thus $G_n$ can be viewed naturally as a subset of $G_{n+1}$ for all $n\geq 0$. 
			\item Every parent-child edge in $G_n$ has length at least $\alpha_n$.
			\item Every vertex in $G_n$ has degree at most 6.
		\end{enumerate}
	\end{theorem}

	\begin{proof} 

	The recursive construction of these graphs occurs explicitly, in generations.

	Generation 0:	
	Starting with the metric graph $(G_0,d_0)$, 
	properties (a), (c) and (d) of the theorem are obvious at this point.

	Generation $n \geq 1$: The rest of the construction is recursive, requiring us to repeatedly augment 
	$G_{n-1}$ to obtain $G_n$. We assume at each generation $n \geq 1$ to have already constructed a connected 
	finite
	metric graph $G_{n-1}$ and the most recently added vertex set $V_{n}$. 
	Now we describe the augmentation procedure, in which we add children to certain pairs
	of vertices in $V_n$.
	In particular, 
	the number of children of a vertex
	$v \in V_n$ starts as 0 and may change multiple times during the augmentation procedure.
	Starting from $G_{n-1}$: 
	\begin{enumerate}[label=(\arabic*)]
		\item Choose a pair of vertices $u,v \in V_n$ that satisfy both of the following properties:
		\begin{enumerate}[label=(\roman*)]
		\item Each of $u$ and $v$ has fewer than 4 children.
		\item The distance between $u$ and $v$ is at least $2\pi -2\alpha_n$ when measured in the path metric
		of the current graph, $d_{\text{cur}}$.
		\end{enumerate}
		\item For such a pair $u,v\in V_n$ we add a new child vertex $w$ joined to each of $u$ and $v$ by an edge
		of length $\alpha_n$.
	\end{enumerate}
	\begin{figure}[h!]
	\begin{center}
	\includegraphics{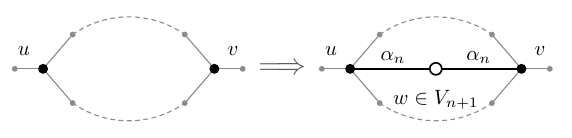}
	\caption{Two vertices $u$ and $v$ in $V_n$ that each have fewer than 4 children and are currently at least 
	distance $2\pi-2\alpha_n$ from each other are joined through a child.}
	\label{augment}
	\end{center} 
	\end{figure}
	Note that doing so 
	may change the path metric by adding shortcuts. After successfully completing Step (2) and potentially changing the
	underlying path metric, go back to Step (1). If it is no longer possible to do Step (1), the augmentation procedure 
	terminates. More importantly, the augmentation procedure always terminates because the 
	degree constraint does not allow us to attach more than 4 edges to each of the finitely many vertices in
	$V_n$. To be clear, no pair of
	vertices in $V_n$ can have more than one common child because condition (ii) in the augmentation procedure
	is never satisfied if two vertices already have a common child (see Figure \ref{augment}).
	To finish the recursive step, 
	we define the next generation of vertices $V_{n+1}$ 
	to be the set of all children added during this generation and declare $G_n$ to be the metric
	graph obtained when the augmentation terminates. In particular, $V(G_n) = \bigcup_{j=0}^{n+1}V_j$.
	
	Clearly $G_n$ is connected since $G_0$ was connected and the augmentation 
	procedure only ever adds edges connected to existing 
	parts of the graph. Thus (b) and (c) are evident from the way we constructed $G_n$.
	Finally, (d) follows because a vertex in $V_n$ initially has degree 2 (it is a child of exactly two parents
	in $G_{n-1}$) 
	and then has at most 4 edges added to it during the augmentation procedure. 
\end{proof}
	
	Some comments are in order. 
	The choice of edge lengths in the initial graph will be used to 
	show that all cycles in $G_0$ have length at least $2\pi$, and inductively ensure that all cycles in 
	the to-be-constructed space $Y$
	have length at least $2\pi$. This will be used in turn to show that $Y$ is CAT(1).
	Also, since Theorem \ref{graphsequence}(b) allows us to naturally view $G_n$ as a subset of $G_{n+1}$, it is clear
	that \[d_{n+1}(x,y) \leq d_n(x,y) \text{ for all } x,y\in G_n\]
	because the addition of new edges means that the $G_{n+1}$ distance between two points in $G_n$ can only decrease.
	It is convenient for future bookkeeping to keep track of an auxiliary sequence of graphs $H_n$ comprising
	the vertex set $V_n$
	with edges between any vertices in $V_n$ that have a common child.  This graph $H_n$ is agnostic to the 
	existing structure of $G_{n-1}$ and could even be disconnected. Note that $|V_{n+1}| = |E(H_n)|$ by definition.
	Now we show that all cycles in the initial graph $G_0$ are sufficiently long.
	\begin{lemma}
		\label{longcycles}
		Every cycle in $G_0$ has length at least $2\pi$.
	\end{lemma}
	\begin{proof}
		We argue based on the number of vertices in $V_1$ contained in each cycle. 
		A cycle in $G_0$ that contains no vertices in $V_1$ must be the 
		initial circle $C$, which has length $5\pi/2 > 2\pi$. A cycle containing
		exactly one vertex in $V_1$ must traverse two parent-child edges of length $\alpha_0$
		and half of the initial circle $C$, giving total length
		\[2\alpha_0 + \frac{5\pi}{4} = \frac{2M\pi}{2M+1} + \frac{5\pi}{4} = \frac{400\pi}{401} + \frac{5\pi}{4} > 2\pi.\]
		Finally, a cycle containing at least two vertices in $V_1$ must traverse 
		at least 4 parent-child edges of length $\alpha_0$, and at least two circle edges of length
		$\delta$, giving total length at least
		\[4\alpha_0 + 2\delta = \frac{4M\pi}{2M+1} +\frac{5\pi}{2M} = \frac{800\pi}{401} +\frac{5\pi}{400} > 2\pi.\qedhere\]
	\end{proof}

	Our next goal is to prove that the sequence of graphs described by Theorem \ref{graphsequence}
	never stops growing. That is, new children are added in each generation $n$ for all $n\geq 1$.
	In fact we will show that an exponential number of new children are added in each generation. 
	For completeness, we record a lemma for graphs of degree at most $\Delta$ that bounds the number of vertices in a ball of 
	combinatorial radius $R$. 
	
	\begin{lemma}
		\label{ballcardinality}
 		In a graph of degree at most $\Delta \geq 3$, the number of vertices in a ball of combinatorial radius at most $R$ is at most
		$\Delta(\Delta-1)^R$.
	\end{lemma}

	\begin{proof}
		Starting from the center of the ball leaves us with at most $\Delta$ choices for the first step, 
	then at most $\Delta-1$ choices for each further non-backtracking step. Then in $R$ steps, the greatest number of vertices one can reach is
	\[1 + \Delta + \Delta(\Delta-1) + \Delta(\Delta-1)^2 + \cdots + \Delta(\Delta-1)^{R-1} \leq \Delta(\Delta-1)^R.\qedhere\] 
	\end{proof}
	
	Now we can prove our desired growth estimate.

	\begin{theorem}
		\label{exponentialvertices}
		For each integer $n\geq 1$ let $N_n = |V_n|$ be the
		 number of vertices in $G_n$ that are children of those in $V_{n-1}$.
		 Then $N_n \geq 50\cdot 2^{n-2}$ for all $n\geq 2$.
	\end{theorem}
	\begin{proof}
		We begin by estimating how many children are created in the first 
		augmentation from parents in $V_1$. Consider two vertices $u,v\in V_1$ 
		that do not have a common child in $V_2$. Their parent sets in $V_0$ are obviously
		disjoint, and shortest paths between $u$ and $v$ can occur in only two ways:
		either taking at least four parent-child edges between $V_1$ and $V_2$, or passing through at least two parent-child
		edges between $V_0$ and $V_1$ plus at least one circle edge. Thus the distance in $G_1$ between $u$ and $v$ is at least
		\[d_{1}(u,v) \geq \min\st{4\alpha_1,2\alpha_0+\delta} =2\alpha_0 + \delta \geq 2\pi - 2\alpha_0,\]
 		where the equality and second inequality can be verified directly from the definitions of $\alpha_0, \alpha_1, \delta$.
		This indicates that $u$ and $v$ would have been joined by a child, if they did not already have
		4 children. Said differently, the preceding argument shows any two vertices in $V_1$ with fewer than 
		4 children must have a common child. Let $U_1\seq V_1$ be the
		set of vertices having fewer than 4 children. By definition of $H_1$, we have just shown
		that $U_1$ is a clique in $H_1$. Then since a vertex in $U_1$ is joined to at most 3 children, this clique
		has size at most 4, i.e. $|U_1| \leq 4$. Then the handshaking lemma gives us
		\[N_2 = |E(H_1)| = \frac{1}{2}\sum_{v\in V_1}\deg_{H_1}(v) \geq 
		\frac{1}{2}\sum_{v\in V_1\setminus U_1}\deg_{H_1}(v) = 2(|V_1| - |U_1|) \geq 
		2(M - 4).\] 
		Now we analyze the future augmentations in generations $n\geq 2$.
		We define the sets $U_n\seq V_n$ similarly as those vertices in $V_n$ 
		having fewer than 4 children. Then for any vertices $u,v\in U_n$,
		it must hold that
		\begin{equation}
			\label{distbound}
			d_n(u,v) < 2\pi -2\alpha_n
		\end{equation}
		because otherwise they would have been eligible to have another child
		before
		terminating the construction of $G_n$. 
		We now argue that this diameter bound \eqref{distbound} on $U_n$ translates to 
		a bound on the combinatorial diameter of $U_n$.
		A path between vertices in $U_n\seq V_n$ that contains an edge of the initial 
		circle $C$ must traverse at least $2n$ parent-child edges because it starts and
		ends in $V_n$, which would give total length at least (recall Theorem \ref{graphsequence}(c))
		\[2n\alpha_n > 2\pi - 2\alpha_n > d_n(u,v) \quad (n\geq 2),\]
		where the first inequality can be verified directly from the definition of $\alpha_n$. 
		Thus such a path could not be shortest, i.e.\ all shortest paths between vertices
		in $U_n$ contain no edges of $C$. Then for any two vertices $u,v$ in $U_n$, letting 
		$P$ be a shortest path in $G_n$ connecting $u$ to $v$ via $k$ edges, we have 
		\[2 \pi - 2\alpha_n > d_n(u,v) = \sum_{e\in P}\text{length}(e) \geq \alpha_n k,\]
		where the first inequality is \eqref{distbound} and the second Theorem \ref{graphsequence}(c).
		Therefore 
		\[k  < \frac{2\pi - 2\alpha_n}{\alpha_n} = 2 + \frac{2n}{M}.\]
		Thus $U_n$ is either empty or contained in a ball of combinatorial radius $R = \left\lfloor 2+\frac{2n}{M}\right\rfloor$ 
		and we may use Lemma \ref{ballcardinality} and Theorem \ref{graphsequence}(d) to infer
		\[|U_n| \leq 6\cdot 5^{R} \leq 6\cdot 5^{2+2n/M}  = 150 q^n \text{ where } q = 5^{2/M}.\]
		Now use the handshaking lemma to lower bound $N_{n+1}$ exactly as we did $N_2$:
		\begin{equation}
			\label{approxdoubling}
			N_{n+1} = |E(H_n)| = \frac{1}{2}\sum_{v\in V_n}\deg_{H_n}(v) \geq 2(|V_n|-|U_n|) \geq 2(N_n - 150q^n).
		\end{equation}
		Intuitively, \eqref{approxdoubling} represents that the number of children $N_{n+1}$ is roughly double the number
		of parents $N_n$, up to an error that grows with geometric rate $1< q < 2$. Since the ratio $q/2$ is less than 1, 
		the relative contribution of this error
		 over the entire infinite construction is summable and thus 
		the growth rate of $N_n$ is at least a power of $2$, provided there are enough parents at the start to overpower the error terms. 
		The subsequent argument
		makes this precise.
		An easy induction starting at $n = 2$ gives the first inequality below, followed by 
		applying the bound $N_2 \geq 2(M-4)$ and replacing a sum with a convergent series: 
		\begin{align}
			\label{doublingbd}
			\frac{N_{n}}{2^{n-2}} &\geq N_2 - \sum_{k=2}^{n-1}\frac{150q^k}{2^{k-2}}\\
			&\geq 2(M-4) - 150 q^2\sum_{k=0}^{\infty}\left(\frac{q}{2}\right)^k\nonumber \\
		&= 2(M-4) -\frac{150 \cdot 5^{4/M}}{1-5^{2/M}/2} \quad (n\geq 2)\nonumber.
		\end{align}
		Clearly the right-hand side goes to $\infty$ as $M\to \infty$
		and in particular is eventually positive, and this is more than 
		sufficient to see that $N_n$ is always positive. Our choice of $M = 200$ reflects one
		explicit option which makes the right-hand side greater than 50, giving the 
		quantitative version in the theorem. To verify this, two applications of Bernoulli's inequality
		give
		\[q = 5^{1/100} = (1+4)^{1/100} \leq 1 + \frac{4}{100} = \frac{26}{25}\]
		\[q^2 = (1+4)^{1/50}\leq 1+\frac{4}{50} = \frac{27}{25}.\]
		It follows that
		\[\frac{150q^2}{1-q/2}\leq \frac{150\cdot (27/25)}{1-13/25} = \frac{675}{2},\]
		which gives
		\[2(M-4) -\frac{150q^2}{1-q/2} \geq 392 - 675/2 = \frac{109}{2} > 50.\] 
Substituting into \eqref{doublingbd} yields the result. \qedhere
	\end{proof}

	Combined with Theorem \ref{graphsequence}, we obtain a sequence of metric 
	graphs $(G_n,d_n)$ that are strictly increasing in the sense that $G_n \subsetneq G_{n+1}$ for all $n\geq 0$, with 
	different metrics satisfying $d_n\geq d_{n+1}$ when both metrics are restricted to $G_n$.
	Forming the basis of the construction is the limiting space
	\[Y = \bigcup_{n=0}^\infty G_n\]
	with pseudometric $d$ defined by
	\[d(x,y) = \lim_{n\to \infty}d_n(x,y) \text{ for any } x,y \in Y.\]
	This is well-defined because $x,y$ both belong to a common $G_N$ for some $N\in \N$, and thus belong to $G_n$ for all $n \geq N$, 
	and the limit exists because it is bounded below and nonincreasing.
	It is easily checked that $d$ is a pseudometric and with a bit more work (Proposition \ref{metricgraph})
	we will see that it is actually a metric. One simple observation is that
	each $G_n$ is a compact subset of $Y$ because $(G_n,d_n)$ is a finite metric graph
	and thus compact, while the inclusion $(G_n,d_n) \hookrightarrow (Y,d)$ is continuous since
	$d\leq d_n$.

	To better measure the geometric structure of $Y$ we define a nonnegative function $h\from Y\to \R_+$  
	that captures in some sense the ``height'' of a point in $Y$, in a continuous way, as it emerges in the 
	increasing construction of $Y$ from the graphs $G_n$. First, set $h \equiv 0$ on the initial circle $C$.
	Then, recalling the nonnegative scalars $\st{s_n}_{n=0}^\infty$ from \eqref{scalars}, define $h$ on each vertex set $V_n$, $n\geq 0$, by
	$h(v) = s_n$ whenever $v\in V_n$. Note that this is consistent with $h\equiv 0$ on $C$. 
	Then, for each $n\geq 0$, extend $h$ to all points 
	on the edges connecting $V_n$ to $V_{n+1}$ by 
	linearly interpolating the endpoint values $s_n$ and $s_{n+1}$. 
	Figure \ref{heightfunct} provides an intuitive sketch of how the height function is defined and distinguishes
	the different generations of $Y$.

	\begin{figure}[h!]
	\begin{center}
		\includegraphics[scale=0.9]{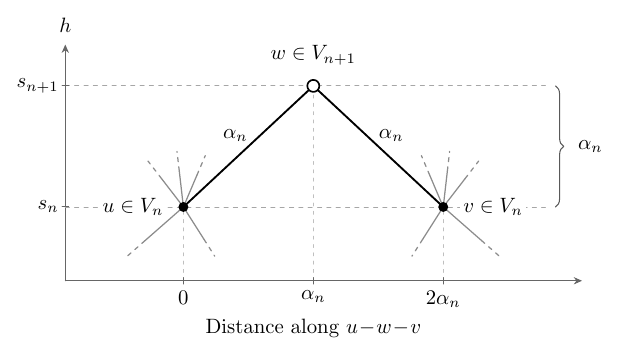}
	\end{center}
	\caption{Height interpolates between values $s_n$ and $s_{n+1}$ on parent-child edges.}
	\label{heightfunct}
\end{figure}

	\begin{lemma}
		\label{h-Lipschitz}
		The height function $h\from Y\to \R_+$ is 1-Lipschitz with respect to $d$.
	\end{lemma}

	\begin{proof}
		Consider any two points $x,y\in Y$ and choose $N \in \N$ sufficiently large such that 
		$x,y$ belong to $G_n$ for all $n\geq N$. For any $n\geq N$, $G_n$ is a finite metric graph 
		and for two points $x,x'\in G_n$ belonging to a common edge we denote by $d_{n,\text{edge}}(x,x')$ the length
		of the edge segment between $x,x'$ in $G_n$.
		Since $G_n$ is finite with $d_n$ its path metric, there is a path in 
		$G_n$ consisting of finitely many edge segments $[x_0,x_1]\cup[x_1,x_2]\cup\cdots\cup[x_k,x_{k+1}]$ 
		with $x_0 = x, x_{k+1} = y$ that satisfies
		$d_n(x,y) = \sum_{i=0}^{k}d_{n,\text{edge}}(x_i,x_{i+1})$. It follows from the triangle inequality in $\R$ that
		\[|h(x) - h(y)| \leq \sum_{i=0}^k|h(x_i)-h(x_{i+1})| 
		\leq \sum_{i=0}^kd_{n,\text{edge}}(x_i,x_{i+1})
		= d_n(x,y),
		\]
		where the second inequality relies on the piecewise linear definition of $h$ on the edges of $G_n$.
		Let $n\to \infty$ to conclude.
	\end{proof}
	It is somewhat nonobvious that the pseudometric $d$ on $Y$, defined by a pointwise limit of graph metrics, 
	coincides with the path length pseudometric on $Y$, but this is true and we will verify it. 
	With infinite graphs it can sometimes occur that this is really a pseudometric and not a metric
	but the simple structure of $Y$ precludes this undesirable behavior: each vertex in $Y$ has finitely many
	incident edges, and in particular there is no vertex whose incident edge lengths approach zero.
	This is enough to guarantee that the length pseudometric on a metric graph is
	actually a metric \cite[I.1.9(1)]{Bridson1999}. 
	The next proposition makes this informal discussion rigorous.
	
	\begin{proposition}
		\label{metricgraph}
		The space $(Y,d)$ is a metric graph and a proper length space.
	\end{proposition}
	\begin{proof}
		The underlying graph structure of $Y$ is clear: the set of vertices is the union of all the vertex sets $V_n$, $n\geq 0$, and the set
		of edges is the union of all the edges joining parents in $V_{n-1}$ to children in $V_{n}$, $n\geq 1$, plus those of the initial circle $C$.
		To make $Y$ a metric graph it must be equipped with the path length pseudometric $d'$ defined by taking the infimum of lengths of piecewise linear paths 
		between two given points. Therefore we must check that $d = \lim_{m\to\infty}d_m$ agrees with this pseudometric $d'$.
		It is clear that $d' \leq d$ because if $x,y\in Y$ belong to some $G_M$ then
		$d'(x,y)\leq d_m(x,y)$ for all $m\geq M$ since every piecewise linear path in $G_m$ is also such a path in $Y$. Letting $m\to \infty$ proves $d'\leq d$.
		On the other hand, any piecewise linear path $P$ in $Y$ between two points $x,y$ must belong to some $G_m$
		since such a path, by definition, consists of 
		finitely many edge segments. Thus
		\[d(x,y)\leq d_m(x,y) \leq \text{length}(P).\]
		Taking the infimum over such paths $P$ shows $d\leq d'$ and thus $d = d'$. 
		Since every vertex in $Y$ has at most 6 incident edges (Theorem \ref{graphsequence}(d)), 
		the condition of each vertex having a positive lower bound on its incident edge lengths is satisfied 
		and thus $d$ is a metric and $Y$ is a metric graph. 
		Along with \cite[Lemma I.5.20]{Bridson1999}, we see that $(Y,d)$ is a length space.

		Finally, we show that $(Y,d)$ is proper. We must prove that an arbitrary closed metric ball $B_R(x)$ in $Y$ is compact.
		Lemma \ref{h-Lipschitz} shows that for any $z\in B_R(x)$,
		\[h(z) \leq h(x) + d(x,z) \leq h(x) + R.\]
		Thus $B_R(x)$ is contained in the closed sublevel set $K = \st{h \leq h(x)+ R}$, and this sublevel 
		set is in turn contained in some $G_n$ because the
		sequence of heights $\st{s_n}_{n=0}^\infty$ tends to infinity.
		We know that $G_n$ is compact in $Y$,
		thus so is $B_R(x)$.
	\end{proof}
 	Thanks to Proposition \ref{metricgraph}, we may apply the Hopf-Rinow theorem \cite[Proposition I.3.7, Corollary I.3.8]{Bridson1999} to conclude:
	\begin{corollary}
		\label{completegeo}
		The space $(Y,d)$ is a complete geodesic space.
	\end{corollary}

	The last key step in this section is to establish that $Y$ is CAT(1). 
	
	\begin{theorem}
		\label{cat1}
		The metric graph $(Y,d)$ is a complete CAT(1) space.
	\end{theorem}
	\begin{proof}
		Completeness was part of Corollary \ref{completegeo}.
	The generalized 
	Cartan-Hadamard theorem \cite[Theorem 9.67]{alexander2024alexandrov} asserts that it 
	suffices to show $Y$ is locally CAT(1) and every cycle in $Y$ has length at least $2\pi$
	(see also \cite{petterson2021graphcat}).
	In fact it is easy to see that $Y$ is locally CAT(0): this is evident at any point in the interior of an edge, and every vertex has degree at most 6 
	by Theorem \ref{graphsequence}(d) so a sufficiently small neighborhood of a vertex is a tree. CAT(0) spaces are also CAT(1) so $Y$ is locally CAT(1).
	
	Any cycle $C'$ in $Y$ is compact as the continuous image of the compact set $S^1$. As such, the continuous height function
	$h$ is bounded above on $C'$ and thus $C'$ must belong to some $G_n$ as we have argued before. Since $G_n$ has only finitely many edges, the cycle
	$C'$ must traverse only finitely many edges in $Y$. Let $N$ be the largest integer such
	that a vertex in $V_N$ belongs to $C'$. If $N \leq 1$ then $C'$ is contained in $G_0$
	and we already know that the length of $C'$ is at least $2\pi$ by Lemma \ref{longcycles}.
	Thus we may assume $N \geq 2$ and we must show that the length of $C'$ is at least $2\pi$.

	\textit{Claim 1: $C'$ contains both edges from each vertex in $V_N \cap C'$ to its parents.}

	To see this, consider a vertex $v$ in $V_N\cap C'$ which by definition is incident to an edge 
	$e$ of $C'$. 
	By the way the augmentation procedure works, 
	the only edges incident to $v$ that do not connect to vertices in $V_{N+1}$
	are those to its parents. But $V_N$ is the last set of vertices represented among the vertices of $C'$, so
	$C'$ cannot leave $v$ through any edges from $v$ to children in $V_{N+1}$. Since $C'$ is a cycle and contains $e$, it must 
	contain the other parent edge of $v$ to avoid being stuck in a dead end.
	
	\textit{Claim 2: The length of $C'$ is at least $2\pi$.}

	In the construction of $G_{N-1}$ from $G_{N-2}$ 
	there is some vertex $w$ in $V_N$ that was the last vertex in $V_N$ to become part of the cycle $C'$.
	By Claim 1, $C'$ contains the edges of length $\alpha_{N-1}$ from $w$ to both of its parents in $V_{N-1}$. 
	Denote the parents
	by $u,v\in V_{N-1}$. If we remove the edge segments $(u,w]$ and $[w,v)$ from $C'$ then what remains
	is a path $P$ from $u$ to $v$. 
	Now since $w$ was the last vertex in $V_N$ added to become part of $C'$ during the augmentation 
	procedure,
	and $N$ is the largest
	index of vertices in $C'$, 
	every other edge in $C'$ was already present when the child $w$ was added, and thus $P$ was a valid path from $u$ to $v$ at that time.
	By condition (ii) in 
	the augmentation procedure in the proof of Theorem \ref{graphsequence}, it must have held that every path 
	from $u$ to $v$ before the addition of
	$w$ had length at least $2\pi -2\alpha_{N-1}$. Hence
	\begin{align*}
		\text{length}(C') &= \text{length}(P) + \text{length}([u,w])+\text{length}([w,v])\\
		& \geq 2\pi - 2\alpha_{N-1} + \alpha_{N-1}+ \alpha_{N-1}\\
		&= 2\pi.\qedhere
	\end{align*}
	\end{proof}

	\subsection{The Hadamard space and the three-point set}
	\label{finiteset}
	The technical setup in the previous section will begin to rapidly pay off. But first let us recall one more standard construction in Alexandrov geometry, namely that 
	of the Euclidean cone over a metric space (we follow precisely the definition in \cite[Definition I.5.6]{Bridson1999}). Given a metric space $(Y,d)$, the \textit{Euclidean cone}
	over $Y$, denoted by $X = C_0Y$, is the metric space defined as follows. As a set, $X$ is the quotient of $[0,\infty)\times Y$ by the 
	equivalence relation $(t,y) \sim (t',y')$ if $(t = t' = 0)$ or $(t=t' > 0\text{ and } y = y')$. The equivalence class of $(t,y)$ is denoted by $[t,y]$, 
	while the equivalence class of $(0,y)$ is denoted by $0$ and called the \textit{tip} of the cone.
	The metric $d_0$ on $X$ is defined by
	\[d_0^2([t,y],[s,z]) =  t^2 + s^2-2ts\cos\left(\min\st{\pi,d(y,z)}\right).\]
	It is an immediate consequence of how the cone metric is defined that for a geodesic in $Y$ of length $L \leq \pi$, say $\gamma\from [0,L]\to Y$,
	the subcone $C_0\gamma \seq C_0Y$ is isometric to a Euclidean sector of angle equal to the
	length $L$, and thus convex (see also \cite[Proof of Proposition I.5.10(1)]{Bridson1999}). Specifically, the isometry
	sends a point in the sector described by polar coordinates $(r,\theta) \in [0,\infty)\times [0,L]$ to 
	$[r,\gamma(\theta)]\in X$.
	Using this isometry, it is easily shown by elementary geometry that the midpoint of $[r,\gamma(0)]$ and $[r,\gamma(L)]$ is exactly the point
	$[r\cos(L/2),\gamma(L/2)]$ as illustrated by Figure \ref{sectorpic}.

	\begin{figure}[h!]
		\centering
		\includegraphics[scale=1]{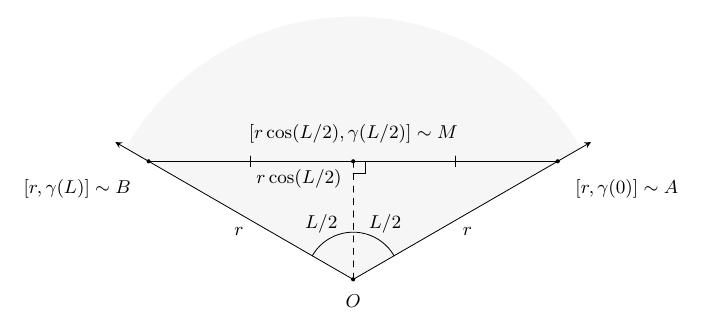}
		\caption{The cone over a geodesic $\gamma$ of length $L \leq \pi$ is isometric to a convex
		Euclidean sector. The Euclidean points $A,B,M$ in the sector are labeled according to the corresponding cone point that 
		they are mapped to under the natural isometry.}
		\label{sectorpic}
	\end{figure}

	By Theorem \ref{cat1}, the space $(Y,d)$ is a complete CAT(1) space, so Berestovskii's theorem 
	\cite[Theorem II.3.14, Proposition I.5.9]{Bridson1999} guarantees that the Euclidean cone $X := C_0Y$ is a complete CAT(0) space, 
	i.e.\ a Hadamard space.
	The initial three-point set will correspond to 
	three equally spaced points along the initial circle $C$ in $G_0$, and a short preliminary argument will be used to show
	that the convex hull of this set is large enough to generate points 
	corresponding to all of the children added during the construction of
	$Y$. 

	Let $p \from [0,5\pi/2]\to Y$ be a unit-speed parametrization of the 
	cycle $C$ in $G_0$. Consider the three-point set 
	\[A = \st{[1,p(0)], [1,p(5\pi/6)], [1,p(5\pi/3)]} \seq X.\] 
	Define a positive sequence $\st{r_n}_{n=0}^\infty$, thought of as radii, by
	\begin{equation}
		\label{radius}
		r_0 = \cos(5\pi/12) > 0, \quad \quad r_{n+1} = r_n\cos\alpha_n.
	\end{equation}

	\begin{lemma}
		\label{lowerbd}
		The sequence of radii $\st{r_n}_{n=0}^\infty$ admits a positive lower bound:
		\[r_n \geq r_* := r_0 \prod_{j=0}^\infty \cos(\alpha_j) > 0 ~ ~ \text{for all } n\geq 0.\]
	\end{lemma}

	\begin{proof}
		Clearly $r_n \geq r_*$. With calculus one can show $\log(\cos(t)) \geq -t^2$ on the interval $[0,1/2]$. Choosing
		$N$ sufficiently large such that $\alpha_N \leq 1/2$, for any $n\geq N$ we find
		\begin{align*}
			\log\left(r_0\prod_{j=0}^n\cos(\alpha_j)\right)
			&= \log(r_0)+\sum_{j=0}^{N-1}\log(\cos(\alpha_j)) + \sum_{j=N}^n\log(\cos(\alpha_j))\\
			&\geq \log(r_0)+\sum_{j=0}^{N-1}\log(\cos(\alpha_j)) -\sum_{j=N}^n \alpha_j^2.
		\end{align*}
		Letting $n\to \infty$ shows $\log(r_*) > -\infty$ because $\sum_{j=N}^\infty \alpha_j^2 < \infty$.
	\end{proof}

	We will show that the convex hull of $A$ contains the following sets of points:
	\begin{equation}
		\label{sepset}
		Q_n := \st{[r_n,v] : v\in V_n} ~ ~ (n\geq 0).
	\end{equation}
	The sets $Q_n$ are constructed in accordance with a few desired properties. First, the angles $\st{\alpha_n}_{n=0}^\infty$ were chosen to be
	\begin{enumerate}[label=(\arabic*)]
		\item square-summable, so that the radius $r_n$ assigned to points in $Q_n$ maintains a minimum distance from the tip of the cone (Lemma \ref{lowerbd}), but 
		\item not summable, so that the heights $\st{s_n}_{n=0}^\infty$ tend to infinity. In turn, these heights in $Y$ translate to 
		angles in $X$, and points in $Q_{n}$ gradually become uniformly separated (in the cone metric) from points in $Q_{n'}$ by making $n \gg n'$. 
	\end{enumerate}
	This is reminiscent of a strategy used by Monod in constructing a bounded Hadamard 
	space with no extreme point \cite{Monod2016}.
	For us, these properties are essential in proving the following lemma:
	\begin{lemma}
		\label{separation} 
		The set of points $Q := \bigcup_{n=0}^\infty Q_n$ is not totally bounded.
	\end{lemma}

	\begin{proof}
		The sequence $\st{s_n}_{n=0}^\infty$ tends to infinity so we can find a subsequence
		$\st{s_{n_j}}_{j=0}^\infty$ such that
		\[s_{n_{j+1}} - s_{n_j} \geq \pi \quad (j \geq 0).\]
		Theorem \ref{exponentialvertices} implies that $V_{n_j}$ is nonempty for all $j \geq 0$ so for each such $j$ choose a vertex
		$v_j \in V_{n_j}$. This gives rise to a sequence $x_j = [r_{n_j},v_j] \in Q$, $j\geq 0$.
		For any indices $i\neq j$,  Lemma \ref{h-Lipschitz} and the definition of the height function $h$ imply
		\[d(v_i,v_j) \geq |h(v_i) - h(v_j)| = |s_{n_i} - s_{n_j}| \geq \pi.\]
		Therefore the distance between $x_i,x_j$ ($i\neq j$) in $X$ is 
		\[d_0(x_i,x_j) = r_{n_i} + r_{n_j} \geq 2r_* > 0,\]
		where we used Lemma \ref{lowerbd}. We see that all points in the sequence $\st{x_j}_{j=0}^\infty$ are distance at least $2r_* > 0$ apart
		and thus cannot be contained in finitely many balls of radius $r_*/2 > 0$. Thus $Q$ is not totally bounded.
	\end{proof}

	The following lemma justifies the choice of $A$ corresponding to three equally spaced points along $C$.
	It demonstrates that these three points can be interpolated by geodesics to
	generate a substantial initial subset of $\conv(A)$.

\begin{lemma}
		\label{generatorlemma}
		For every point $y \in C$, 
		there exists a number $r\geq r_0$ such that 
		the point $[r,y]$ in $X$ lies in $\conv(A)$.
	\end{lemma}

	\begin{proof}
	First observe that the points $p(i5\pi/6), i = 0,1,2$, partition $C$
	into three local geodesics of equal length $5\pi/6 < \pi$, so 
	these paths are geodesics in the CAT(1) space $Y$ (see \cite[Proposition II.1.4(2)]{Bridson1999}).
	Choose a point $y\in C$
	and one of these three geodesics which contains $y$, call it $\gamma\from[0,5\pi/6]\to Y$
	and fix $t\in [0,5\pi/6]$ such that $\gamma(t) = y$. 
	Note that the subcone 
	$C_0\gamma \seq C_0Y = X$ is isometric to a Euclidean sector $S$ of angle
	$5\pi/6$. 

	Let $a$ and $b$ in $S$ be points that map isometrically to 
	$[1,\gamma(0)]$ and $[1,\gamma(5\pi/6)]$, respectively, in $X$.
	The line segment $L$
	joining $a$ and $b$
	is distance $r_0 = \cos(5\pi/12)$ from the origin 
	(see Figures \ref{sectorpic} and \ref{sectordiag}). 
	In the sector $S$, consider the line segment from the origin to the point
	corresponding to $[1,y]$. This line segment intersects $L$ at a radius
	$r \geq r_0$ and the point of intersection is mapped isometrically to 
	$[r,y]$ in $X$ (see Figure \ref{sectordiag}). Moreover, this point of intersection is a convex combination of
	$a$ and $b$ and by isometry is a convex combination of $[1,\gamma(0)]$ and $[1,\gamma(5\pi/6)]$,
	i.e.\ belongs to $\conv(A)$. The point $y\in C$ was arbitrary, so we are done. 
	\begin{figure}[h!]
		\centering
		\begin{minipage}[t]{0.57\textwidth}
			\centering
			\vspace{0pt}
			\includegraphics[width=\linewidth]{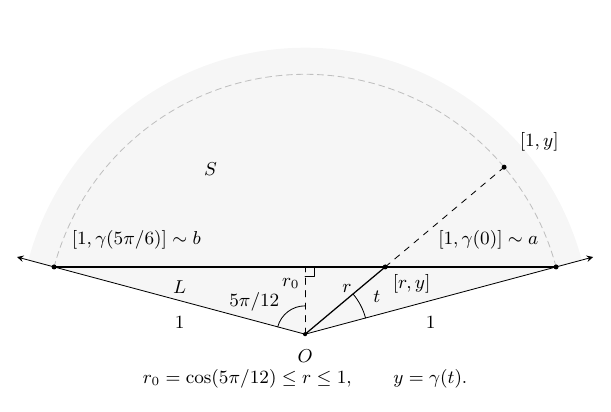}
		\end{minipage}
		\begin{minipage}[t]{0.43\textwidth}
			\centering
			\vspace{0pt}
			\includegraphics[width=\linewidth]{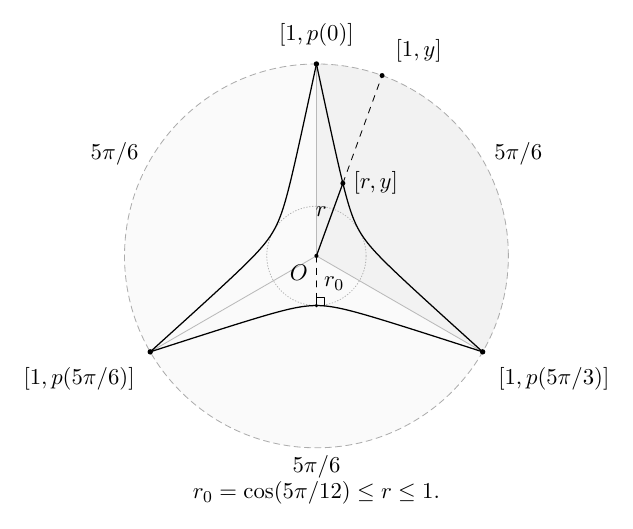}
		\end{minipage}
		\caption{Left: The point $[r,y]$ viewed as a convex combination of $a$ and $b$ in $S$.
		Right: The geodesics between the three points in $A$ generate a triangle in $\conv(A)$ surrounding the cone tip.
		The curvature in the lines connecting the three points in $A$ is an artifact of 
		rescaling a circle of circumference $5\pi/2$ to $2\pi$ by scaling all angles by a factor of $4/5$.}
		\label{sectordiag}
	\end{figure}
	\end{proof}

	Next we show that the tip of the cone $X$ also belongs to $\conv (A)$. In conjunction with
	the previous lemma, this will be used to establish that the set $Q_0$ belongs to
	$\conv(A)$ by interpolating between the tip and the points $[r,y]$ with $r\geq r_0$ and $y\in C$.
	\begin{lemma}
		\label{conetip}
		The tip of the cone $X = C_0Y$ belongs to $\conv(A)$.
	\end{lemma}

	\begin{proof}
		By Lemma \ref{generatorlemma} there exists $r \geq r_0$ such that
		$[r,p(\pi)]$ belongs to $\conv(A)$. Since $p(0)$ and $p(\pi)$ are 
		distance $\pi$ apart, the distance between
		$[1,p(0)]$ and $[r,p(\pi)]$ is 
		\[d_0([1,p(0)],[r,p(\pi)]) = \sqrt{1 + r^2 - 
		2r\cos(\min\st{\pi,d(p(0),p(\pi))})} = 1+r.\]
		The path obtained by concatenating
		the segments from $[1,p(0)]$ to $0$ to $[r,p(\pi)]$ also has length $1+r$, so it is the unique geodesic
		between its endpoints.
		Thus $0$ lies on the geodesic between those endpoints, meaning it belongs to
		the convex hull of $[1,p(0)]$ and $[r,p(\pi)]$. Since these points are
		already in $\conv(A)$, we conclude that the tip $0$ is in $\conv(A)$.
	\end{proof}

	\begin{proposition}
		\label{conetip2}
		For each point $y\in C$, the point $[r_0,y]$ belongs to 
		$\conv(A)$. In particular, since $V_0 \seq C$, we see that $Q_0 \seq \conv(A)$.
	\end{proposition}

	\begin{proof}
		For any point $y\in C$, use Lemma \ref{generatorlemma}
		to choose
		 $r\geq r_0$ such that $[r,y]$ belongs to $\conv(A)$. 
		The cone tip $0$ also belongs to $\conv(A)$ by Lemma
		\ref{conetip}, and the point $[r_0,y]$ lies on the geodesic
		from $[r,y]$ to $0$ (see Figure \ref{sectordiag}), so the result follows.
	\end{proof}

	Finally, we can put all of the pieces together by showing that the set $\conv(A)$ contains the set $Q$. The proof is essentially an 
	exercise in using the geometry of the cone metric to see how the
	midpoints created in the construction of $Y$ become angular midpoints in the Euclidean cone $X$.

	\begin{theorem}
		\label{counterexample}
		In the Hadamard space $X = C_0Y$, the convex hull of the three-point set $A$ is not totally bounded 
		and thus $\overline{\conv}(A)$ is not compact.
	\end{theorem}

	\begin{proof}
		By Lemma \ref{separation} it suffices to show that $\conv(A)$ contains the set $Q_n$ for each $n\geq 0$.
		Proposition \ref{conetip2} proves the base case $Q_0\seq \conv(A)$. Proceeding inductively, assume 
		$Q_n \seq \conv(A)$ and choose an arbitrary vertex $w \in V_{n+1}$. To complete the induction we must show 
		$[r_{n+1},w]$ belongs to $\conv(A)$. By definition of the vertex set $V_{n+1}$ we know that 
		$w$ is the child of two parents 
		$u,v\in V_n$ to which it is joined by edges of length $\alpha_n$. 
		The induction hypothesis ensures $[r_n,u], [r_n,v]$ belong to $\conv (A)$ so we
		need only verify that $[r_{n+1},w]$ is their geodesic midpoint in $X$. 

		Note that the path $P$ from $u$ to $v$ via $w$ is a local geodesic of length
		$2\alpha_n \leq \pi$ in the CAT(1) space $Y$, so it is in fact a geodesic.
		Then the subcone $C_0P \seq C_0Y$ is isometric to a Euclidean sector of
		angle $2\alpha_n$, and the midpoint of the corresponding
		geodesic endpoints $[r_n,u]$, $[r_n,v]$ is exactly $[r_n\cos(d(u,v)/2),w]$
		(see Figure \ref{sectorpic} and the discussion preceding it). 
		Given that $d(u,v) = 2\alpha_n$ we see that this midpoint is precisely 
		$[r_n\cos\alpha_n,w] = [r_{n+1},w]$ by definition of $r_{n+1}$.
		Thus $Q_{n+1} \seq \conv (A)$ and the induction is complete.
	\end{proof}

	\section{Declaration of AI Use}
	The original example was found by a single prompt to ChatGPT Sol 5.6 on Ultra mode. Further interaction with ChatGPT Sol
	5.6 was carried out to verify its correctness and gain understanding of the construction. The refined three-point
	example was found by further prompting ChatGPT Astra on Extra High mode. Again, extended interaction 
	between the author and the AI was used to simplify the smaller example's presentation.
	The author has carefully expanded and clarified the AI-generated claims to obtain the paper in its current form. 
	The contents of the proofs follow closely the strategy laid out by the AI, and interaction with the AI 
	contributed to the author's ability to properly justify the arguments. The author chose specific concepts to illustrate and
	used AI to generate the figures. The actual writing and organization of the paper 
	have been done without the use of AI, with the intent of prioritizing understanding and readability. 
	The author takes full responsibility for its contents.
	\bibliographystyle{plain}
	\small
	\bibliography{ref}
\end{document}